\documentclass{amsart}

\usepackage{rotating}
\usepackage[utf8]{inputenc}
\usepackage{bbm}
\usepackage{graphicx}
\usepackage{dsfont}
\usepackage{color}
\usepackage[hyperref]{hyperref}
\usepackage{amssymb}

\newtheorem{theorem}{Theorem}[section]

\newtheorem{lemma}[theorem]{Lemma}
\newtheorem{proposition}[theorem]{Proposition}
\newtheorem{corollary}[theorem]{Corollary}

\theoremstyle{definition}
\newtheorem{definition}[theorem]{Definition}

\theoremstyle{remark}
\newtheorem{remark}[theorem]{Remark}

\numberwithin{equation}{section}

\newcommand{\abs}[1]{\lvert#1\rvert}

\newcommand{\C}{{\mathbb{C}}}
\newcommand{\R}{{\mathbb{R}}}

\newcommand{\Z}{{\mathbb{Z}}}

\renewcommand{\epsilon}{\varepsilon}
\renewcommand{\phi}{\varphi}
\renewcommand{\theta}{\vartheta}

\newcommand{\lie}[1]{{\mathcal{L}_{#1}}}

\newcommand{\w}{\wedge}

\newcommand{\lcan}{{\lambda_{\mathrm{can}}}}

\DeclareMathOperator{\id}{Id}
\DeclareMathOperator{\ind}{index}

\DeclareMathOperator{\open}{Open}

\begin{document}
\title{Lecture notes on stabilization of contact open books}

\author{Otto van Koert}

\address[O.~van Koert]{Department of Mathematical Sciences\\
  Seoul National University\\
  San56-1 Shillim-dong Kwanak-gu\\
  Seoul 151-747\\
  Korea
}
\email[O.~van Koert]{okoert@snu.ac.kr}




\begin{abstract}
This note explains how to relate some contact geometric operations, such as surgery, to operations on an underlying contact open book.
In particular, we shall give a simple proof of the fact that stabilizations of contact open books yield contactomorphic manifolds.

Let us remark that the results in this note are all well known to experts.
This note just aims to provide some references for these results.
\end{abstract}

\maketitle

\section{Introduction}
The correspondence between open books and contact structures as established by Giroux \cite{Giroux:ICM} has been extremely fruitful in understanding contact structures both in dimension $3$ and in higher dimensions.

In general, this correspondence looks as follows.
Given a Weinstein manifold $W$ and a symplectomorphism $\psi$ of $W$ that is the identity near $\partial W$, we can endow the mapping torus of $(W,\psi)$ with a natural contact form.
The boundary of this mapping torus is diffeomorphic to $\partial W \times S^1$, which allows us to glue in a copy of $\partial W\times D^2$.
The latter set can be given a contact form which glues nicely to the one on the mapping torus.

Conversely, every compact coorientable contact manifold can in fact be obtained by this construction.
However, supporting open books for contact manifolds are not unique.
For instance, one has a stabilization procedure, which does not change the contact structure, but it does change the open book.
Suppose we are given a contact open book $\open(W,\psi)$ with a Lagrangian disk $L$ in a page $W$ such that $\partial L$ is a Legendrian sphere in $\partial W$.
We obtain a new page $\tilde W$ by attaching a symplectic handle to $W$ along $\partial L$.
The monodromy $\psi$ can be extended as the identity on the symplectic handle.
Since $\tilde W$ contains a Lagrangian sphere formed by $L$ and the core of the symplectic handle, we can compose the monodromy $\psi$ with a right-handed Dehn twist $\tau_L$ along this Lagrangian sphere.
This leads us to the (positive) stabilization of $\open(W,\psi)$, which is given by $\open(\tilde W,\tau_L \circ \psi)$.
According to Giroux the stabilization is contactomorphic to the original contact manifold.

In dimension $3$ the above correspondence is even better.
Giroux has shown that on a compact, orientable $3$-manifold $M$, open books for $M$ up to (positive) stabilization correspond bijectively to isotopy classes of contact structures on $M$.

The goal of this note is the clarify some of these well-known notions and to provide proofs for some of them.
We shall discuss the relation between contact surgery and open books:
subcritical handle attachment along isotropic spheres in the binding can be seen as handle attachment to the page of the open book, whereas Legendrian surgery along a Legendrian sphere $L$ in a page can be seen as composing the initial monodromy with a right-handed Dehn twist along $L$.
This implies the well-known assertion that contact open books whose monodromy is isotopic to the product of right-handed Dehn twists are Stein fillable.
We also provide a proof of the fact that stabilization does not change the contact structure.

Our proof is rather elementary and works almost entirely in the contact world. 
In particular, we shall not use Lefschetz fibrations (which could be used to look at the situation from another point of view): we basically interpret the handle attachment to the page and change of monodromy as successive contact surgeries which cancel each other.
To see the latter though, we use symplectic handle cancellation.\\
\\
\noindent{\bf Acknowledgements. }
I thank F.~Ding, H.~Geiges, and K.~Niederkr\"uger for helpful comments and suggestions.

\section{Weinstein manifolds and open books}

\subsection{Weinstein and Stein}

Let us first define the notion of Weinstein manifold.

\begin{definition}
  Let $M$ be a smooth manifold, and let $f:\, M\to \R$ be a smooth
  function.  A vector field $X$ on $M$ is called
  \textbf{gradient-like} for the function $f$ if $\lie{X} f > 0$
  outside the critical points of $f$.
\end{definition}

\begin{definition}
  Let $(W,\omega)$ be a symplectic manifold. A proper function $f:\, W
  \to [0,\infty[$ is called \textbf{$\omega$--convex} if it admits a
  complete gradient-like Liouville vector field $X$, i.e.~$\lie{X}
  \omega = \omega$.  We say $(W,\omega)$ is a \textbf{Weinstein
    manifold} if there exists an $\omega$--convex Morse function.
\end{definition}

\begin{remark}
From this definition it follows that all ends of a Weinstein manifold $W$ are convex, i.e.~they look like symplectizations.
Indeed, let $f$ be an $\omega$--convex function and $X$ be a complete gradient-like Liouville vector field for $f$.
Since the vector field $X$ is assumed to be complete, $W$ cannot have have any boundary components, because critical points of $X$ must be isolated by the Morse condition.

Furthermore, since the Liouville vector field $X$ is gradient-like for $f$, we see that $X$ is positively transverse to regular level sets of $f$. Combined with properness of $f$ this implies convexity of the ends.
\end{remark}

\begin{remark}
We shall also apply the definition of $\omega$--convex function to general symplectic cobordisms.
In such a case the function $f$ may not be bounded from below.
The most basic example is a symplectization $(\R\times M,\omega=d(e^t\alpha) )$, where the function $f(t,x)=e^t$ is $\omega$--convex for $X=\frac{\partial}{\partial t}$.
\end{remark}

Note that $i_X \omega$ defines a primitive of $\omega$, so Weinstein
manifolds are exact symplectic.  For the sake of completeness, let us
briefly recall some related notions.

\begin{definition}
  Let $(W,J)$ be an almost complex manifold.  A function $f:\, W \to
  \R$ is said to be \textbf{strictly plurisubharmonic} if
  \begin{equation*}
    g(X,Y) := -d(df \circ J)(X, JY)
  \end{equation*}
  for all vectors $X,Y$ defines a Riemannian metric.
\end{definition}

Let $(W,J)$ be a Stein manifold.
By a theorem of Grauert, $(W,J)$ admits a strictly plurisubharmonic function $f$.
Denote the associated symplectic form $-d(df\circ J)$ on $W$ by $\omega_f$.  
We then see that strictly plurisubharmonic functions on Stein manifolds are examples of $\omega_f$--convex Morse functions, i.e.~Stein manifolds are Weinstein.
Indeed, by solving the equation
$$
i_X\omega_f=-df \circ J,
$$
we obtain a Liouville vector field that is gradient-like for $f$, as $0 \geq \omega_f(X,JX)=df(X)$.

According to Eliashberg the converse is also true, but since we are only interested in the exact symplectic structure rather
than the complex structure, we shall formulate everything using Weinstein manifolds.

\begin{remark}
  A \textbf{compact Weinstein manifold} $(\Sigma,\omega)$ is a compact
  symplectic manifold with boundary $K$ that can be a embedded into a
  Weinstein manifold $(W,\omega)$ with an $\omega$--convex function
  $f$ such that $\Sigma$ is given as the preimage $f^{-1}([0,C])$, and such
  that $C$ is a regular level set of $f$.  Note that such a regular
  level set is automatically contact.
\end{remark}

\subsection{Contact open books}

\begin{definition}
  An \textbf{abstract (contact) open book} $(\Sigma, \lambda, \psi)$
  consists of a compact Weinstein manifold $(\Sigma, \lambda)$, and a
  symplectomorphism $\psi:\, \Sigma \to \Sigma$ with compact support
  such that $\psi^*d\lambda = d\lambda$.
\end{definition}

Let us now show that an abstract contact open book corresponds to
a contact manifold with a supporting open book.

By a lemma of Giroux \cite{Giroux:talk} we can assume that $\psi^*
\lambda =\lambda-dh$.  We choose the function $h$ to be positive.  For
completeness, here is the lemma and a proof.

\begin{lemma}[Giroux]
  The symplectomorphism $\psi$ can be isotoped to a symplectomorphism
  $\widehat\psi$ that is the identity near the boundary and that
  satisfies
  \begin{equation*}
    \widehat\psi^*\lambda = \lambda-dh \;.
  \end{equation*}
\end{lemma}
\begin{proof}
  Let us denote the $1$-form $\psi^*\lambda-\lambda$ by $\mu$.  Since
  $d\lambda$ is non-degenerate, we find a unique solution $Y$ to the
  equation $i_Y d\lambda=-\mu$.  The flow of the vector field $Y$
  preserves $d\lambda$, because $\mu$ is closed,
  \begin{equation*}
    \lie{Y} d\lambda = d \iota_Y d\lambda = - d\mu = 0\;.
  \end{equation*}
  Since $\psi$ is the identity near the boundary, $\mu$ and hence $Y$
  vanishes near the boundary.  If we denote the time-$t$ flow of $Y$
  by $\phi_t$, then we see that $\widehat\psi = \psi\circ\phi_1$ is a
  symplectomorphism that is the identity near the boundary.  Note that
  $\lie{Y} \mu = 0$, so $\phi_t^*\mu = \mu$ for all $t$.  We check
  that the difference of the pullback of $\lambda$ and $\lambda$ is
  indeed exact.  We have
  \begin{equation*}
    (\psi \circ \phi_1)^* \lambda -\lambda=\phi_1^*(\mu+\lambda)
    -\lambda= \mu+\phi_1^*\lambda -\lambda \; .
  \end{equation*}
  On the other hand, we can express the difference $\phi_1^*
  \lambda-\lambda$ as
  \begin{equation*}
    \begin{split}
      \phi_1^*\lambda-\lambda& = \int_0^1\frac{d}{dt}
      \phi_t^*\lambda\, dt = \int_0^1 \bigl(\phi_t^* \mathcal L_Y
      \lambda\bigr) \, dt = \int_0^1 \phi_t^* \bigl( i_Y d\lambda + d
      (i_Y\lambda)
      \bigr)\, dt  \\
      &= -\mu + d \int_0^1 \phi_t^*(i_Y\lambda)\, dt \;.
    \end{split}
  \end{equation*}
  Moving $\mu$ to the left-hand-side, we see that $\mu+\phi_1^*
  \lambda -\lambda$ is exact, which shows the claim of the lemma.
\end{proof}

Now we can define
\begin{equation*}
  A_{(\Sigma,\psi)} := \Sigma \times \R / (x,\phi)\sim (\psi(x),\phi+h(x)) \;.
\end{equation*}
This mapping torus carries the contact form
\begin{equation*}
  \alpha = \lambda + d\phi \;.
\end{equation*}
Since $\psi$ is the identity near the boundary of $\Sigma$, a
neighborhood of the boundary looks like
\begin{equation*}
  \bigl(-\frac{1}{2},0\bigr] \times \partial \Sigma \times  S^1 \;,
\end{equation*}
with contact form
\begin{equation*}
  \alpha = e^r\,\lambda|_{\partial \Sigma} + d\phi \;.
\end{equation*}
Denote the annulus $\bigl\{z\in \C\bigm|\, r < \abs{z} < R\bigr\}$ by
$A(r,R)$.  We can glue the mapping torus $A_{(\Sigma,\psi)}$ along its boundary to
\begin{equation*}
  B_\Sigma:=\partial \Sigma \times D^2
\end{equation*}
using the map
\begin{equation*}
  \begin{split}
    \Phi_{\mathrm{glue}}:\, \partial \Sigma \times A(1/2,1)
    & \longrightarrow (-1/2,0] \times \partial \Sigma \times S^1 \\
    \bigl(x; re^{i\phi}\bigr) &\longmapsto \bigl(1/2-r, x, \phi\bigr)
    \;.
  \end{split}
\end{equation*}
Pulling back the form $\alpha$ by $\Phi_{\mathrm{glue}}$, we obtain
\begin{equation*}
  e^{1/2-r}\, \lambda|_{\partial \Sigma} + d\phi
\end{equation*}
on $\Sigma \times A(1/2,1)$,
which can be easily extended to a contact form
\begin{equation*}
  \beta = h_1(r)\, \lambda|_{\partial \Sigma} + h_2(r)\, d\phi 
\end{equation*}
on the interior of $B_\Sigma$ by requiring that
$h_1$ and $h_2$ are functions from $[0,1)$ to $\R$ whose behavior is
indicated in Figure~\ref{fig:functions_binding}; $h_1(r)$ should have
exponential drop-off and $h_2(r)$ should quadratically increase near
$0$ and be constant near $1$.

\begin{figure}[htp]
  \centering
  \includegraphics[width=0.6\textwidth,clip]{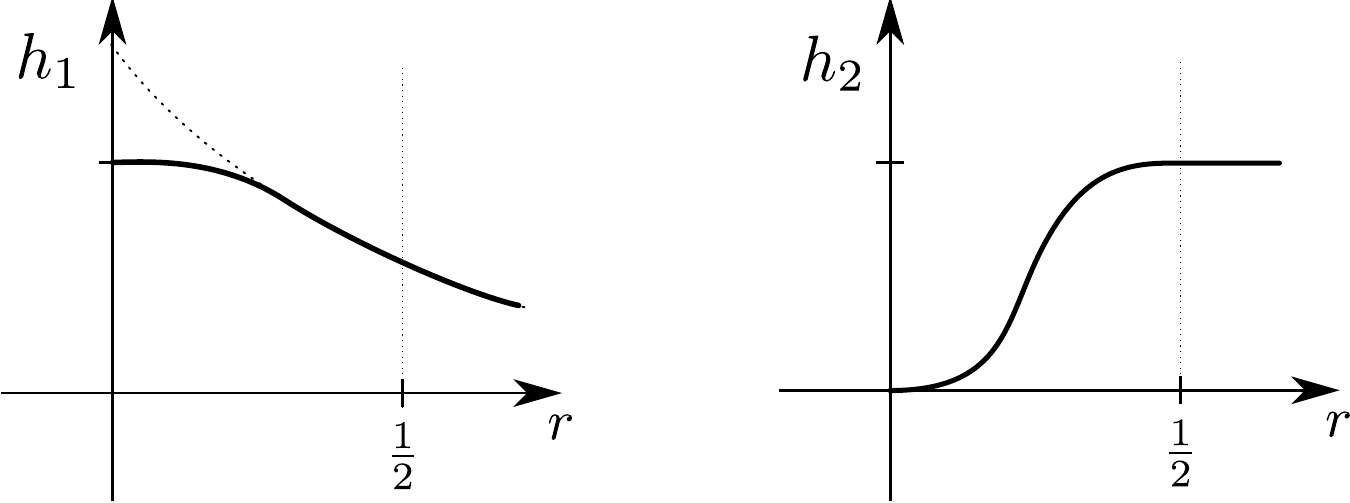}
  
  \caption{Functions for the contact form near the
    binding}\label{fig:functions_binding}
\end{figure}

The union $M := A_{(\Sigma,\psi)} \cup_\partial B_\Sigma$ is called an \textbf{abstract open
  book} for $M$.  Note that the contact forms $\alpha$ on $A_{(\Sigma,\psi)}$ and
$\beta$ on $B_\Sigma$ glue together to a globally defined contact form.

We shall call the resulting contact manifold, which we denote by
$\open( \Sigma,\lambda;\psi)$, a \textbf{contact open book}.
We shall sometimes drop the primitive $\lambda$ of the symplectic form in our later notation.

\begin{remark}
Note that $\open(\Sigma,\lambda;\psi)$ has the structure of a fibration over $S^1$ away from the set $B$.
Hence we can talk about the monodromy of an open book, which can be obtained by lifting the tangent vector field to $S^1$, given by $\partial_\phi$, to a vector field on $A$.
If we rescale the function $h$ to $2\pi$ then the time-$2\pi$ flow gives the monodromy.
Note that a positive function times the Reeb vector field is a suitable lift of $\partial_\phi$.
As a result, we see that the monodromy is given by $\psi^{-1}$.

We should also point out that there are various conventions in use at this point.
Some papers refer to $\psi$ as the monodromy and Milnor \cite{Milnor:singular_points} used the word \emph{characteristic homeomorphism}.
\end{remark}

\begin{definition}
  An \textbf{open book} on $M$ is a pair $(K,\theta)$, where
\begin{itemize}
\item{} $K$ is a codimension $2$ submanifold of $M$ with trivial
  normal bundle, and
\item{} $\theta:\, M-B\to S^1$ endows $M - K$ with the structure of a
  fiber bundle over $S^1$ such that $\theta$ gives the angular
  coordinate of the $D^2$--factor of a neighborhood $B\times D^2$ of
  $B$.
\end{itemize}
\end{definition}

The set $K$ is called the \textbf{binding} of the open book. A fiber
of $\theta$ together with the binding is called a \textbf{page} of the
open book.

\begin{remark}
The typical situation of an open book is the following.
Let $K$ be a knot in a $3$-manifold.
For special knots, so-called fibered knots, the complement fibers over $S^1$ in a nice way: this is equivalent to an open book.
A well-known example is the unknot in $S^3$.
\end{remark}

In order to define the notion of adapted open book, we need to discuss the orientations involved.
Suppose $M$ is an oriented manifold with an open book $(K,\theta)$.
Since we regard $S^1$ as an oriented manifold, each page $\Sigma$ gets an induced orientation such that the orientation of $M-K$ as a bundle over $S^1$ matches the one coming from $M$. 
If this orientation of the page $\Sigma$ matches the orientation as a symplectic manifold, we call a symplectic form $\omega$ on $\Sigma$ {\bf positive}.
We shall orient the binding $K$ as the boundary of a page $\Sigma$ using the outward normal.
If, on the other hand, this orientation matches the one coming from a contact form $\alpha$, i.e.~$\alpha\w d\alpha^n$, then we say that $\alpha$ induces a {\bf positive contact structure}.

\begin{definition}
  A positive contact structure $\xi$ on an oriented manifold $M$ is
  said to be \textbf{carried by an open book} $(B, \theta)$ if $\xi$
  admits a defining contact form $\alpha$ satisfying the following
  conditions.
  \begin{itemize}
  \item{} $\alpha$ induces a positive contact structure on $B$, and
  \item{} $d\alpha$ induces a positive symplectic structure on each
    fiber of $\theta$.
  \end{itemize}
  A contact form $\alpha$ satisfying these conditions is said to be
  \textbf{adapted} to $(B, \theta)$.
\end{definition}

\begin{lemma}
  Suppose that $B$ is a connected contact submanifold of a contact manifold $(M,\xi)$.  A
  contact form $\alpha$ for $(M,\xi)$ is adapted to an open book $(B,
  \theta)$ if and only if the Reeb field $R_\alpha$ of $\alpha$ is
  positively transverse to the fibers of $\theta$,
  i.e.~$R_\alpha(\theta) > 0$.
\end{lemma}
\begin{proof}
  If $d\alpha$ is positive on each fiber of $\theta$, then we can find
  tangent vectors $v_1,\dotsc, v_{2n}$ to the page at a point $x$ such
  that $i_{v_1\w\ldots \w v_{2n}}d\alpha^n>0$.  Hence
  \begin{equation*}
    \iota_{R\w v_1\w\ldots \w v_{2n}}\alpha \w d\alpha^n>0 \;.
  \end{equation*}
  Since the pages and the $S^1$ direction also orient the manifold, we
  see that the Reeb field is positively transverse to the pages.

  Conversely, if $R_\alpha$ is positively transverse to the the fibers
  of $\theta$, then $i_{R_\alpha} \alpha\w d\alpha^n>0$, so in
  particular $d\alpha$ is a positive symplectic form on each fiber.

We assume $B$ to be a contact submanifold, so we only need to check positivity.
Note that
  \begin{equation*}
    \int_{\partial P} \alpha\w d\alpha^{n-1}=\int_P 
    d \alpha\w d\alpha^{n-1} = \int_P d\alpha^{n}>0 \;.
  \end{equation*}
Since the binding $B$ was assumed to be connected, we see that $(B,\theta)$ is a supporting open book.
\end{proof}

\begin{proposition}
  An abstract contact open book $\open(\Sigma, \psi)$ admits a natural
  open book carrying the contact structure $\xi$ in the above
  construction.
\end{proposition}
\begin{proof}
  We define the binding of the abstract contact open book
  $\open(\Sigma, \psi)$ to be the submanifold $\partial \Sigma\times
  \{0\}$.  The map $\theta$ from $M - B$ to $S^1$ can be defined by
  putting $\theta(x)=\phi$ if $x=(p;r,\phi)$ is a point in $\partial
  \Sigma\times D^2$.  For points in $A$, we use the fact the $A$ is a
  fiber bundle over $S^1$.  Moreover, the definitions coincide on the
  overlap of $A$ and $\partial \Sigma\times D^2$.

  The Reeb field of the abstract contact open book $\open(\Sigma,\lambda;\psi)$ as given by the
  above construction is $\partial_\phi$, so it is positively transverse to all pages.
This implies that the open book carries the associated contact structure.
\end{proof}

\subsection{Basic properties of open books}

\subsubsection{Order and monodromy}

In general, the resulting contact manifold depends on the monodromy,
but there are some symmetries.  For instance, if $\Sigma$ is a convex
symplectic manifold and $\psi_1$ and $\psi_2$ are symplectomorphisms,
then
\begin{equation*}
  \open(\Sigma,\psi_2^{-1} \circ \psi_1\circ \psi_2)\cong
  \open(\Sigma,\psi_1) \;.
\end{equation*}
Indeed, we can simply regard the mapping torus of the open book as three products $\Sigma \times I$ glued together.  
Gluing them in another order gives the same result.

This observation also implies the cyclic symmetry property,
\begin{equation*}
  \open(\Sigma,\psi_1\circ \psi_2)\cong
  \open(\Sigma,\psi_2 \circ \psi_1) \;.
\end{equation*}
Indeed, if we conjugate $\psi_2 \circ \psi_1$ by $\psi_2$, we get the above expression.

\subsection{Important examples of monodromies}
In general, the group of symplectomorphisms on a symplectic manifold is poorly understood.
In fact, in many cases, such $(D^6,\omega_0)$, it is unknown whether every symplectomorphism is isotopic to the identity (relative to the boundary).

There is, however, one way to construct candidates of symplectomorphisms that are in general not isotopic to the identity.
Suppose that $(W,\omega)$ is a symplectic manifold with an embedded Lagrangian sphere $L\subset W$.
By the Weinstein neighborhood theorem, a neighborhood $\nu_W(L)$ is symplectomorphic to the canonical symplectic structure on $(T^*S^n,d\lambda_{can})$.

Hence we consider the symplectic manifold $(T^*S^{n},d\lambda_{can})$, where $\lambda_{can}$ is the canonical $1$-form. In local coordinates, this form is given by $\lambda_{can}=p\, dq$.
To describe a so-called Dehn twist, we first regard this manifold as a submanifold of $R^{2n+2}$ by using coordinates
$$
(q,p) \in \R^{2n+2}
$$
subject to the following relations
\begin{equation}
\label{eq:coordinatesTSn}
q \cdot q=1 \text{ and } q \cdot p=0.
\end{equation}
With these coordinates the canonical $1$-form $\lambda_{can}$ on $T^*S^n$ is given by
$$
\lambda_{can}=p \, dq.
$$
Define an auxiliary map describing the normalized geodesic flow
$$
\sigma_t(q,p)=
\left(
\begin{array}{cc}
\cos t & |p|^{-1} \sin t \\
-|p|\sin t & \cos t 
\end{array}
\right)
\left(
\begin{array}{c}
q \\ p
\end{array}
\right) .
$$
Then define
$$
\tau(q,p)=
\left \{
\begin{array}{ll}
\sigma_{g_1(|p|)}(q,p) & \text{ if } p\neq0 \\
-\id & \text{ if } p=0.
\end{array}
\right.
$$
Here $g_1$ is a smooth function with the following properties.
\begin{itemize}
\item{} $g_1(0)=\pi$ and $g_1'(0)<0$.
\item{} Fix $p_0>0$. The function $g_1(|p|)$ decreases to $0$ at $p_0$ after which it is identically $0$.
\end{itemize}
Note that the conditions imply that $\tau$ is actually a smooth map. 
See Figure~\ref{fig:Dehn_twist}.
The map $\tau$ is called a {\bf (generalized) right-handed Dehn twist}.

\begin{figure}[htp]
 \centering
 \includegraphics[width=0.4\textwidth,clip]{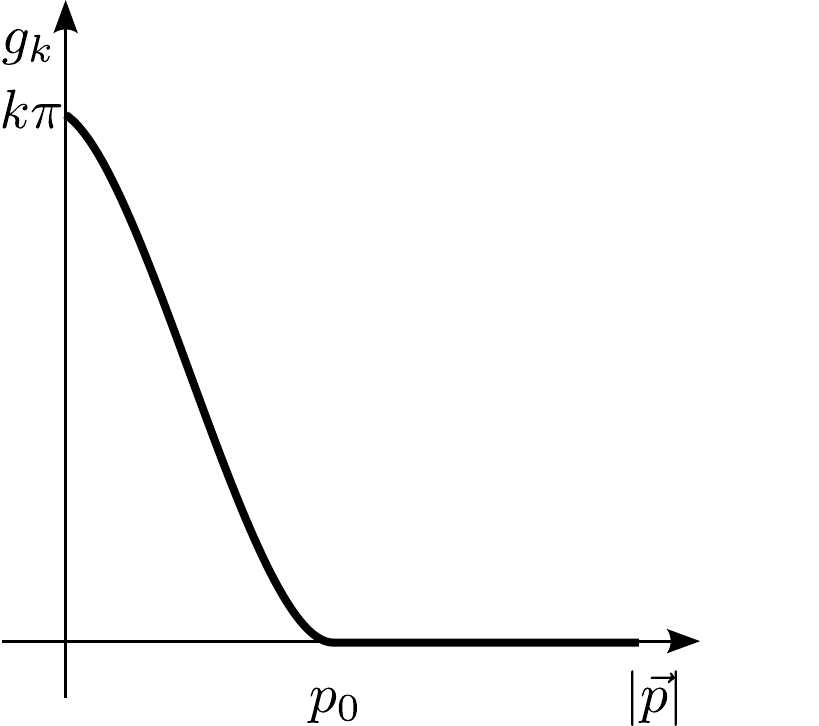}%
 \caption{The amount of geodesic flow for a $k$-fold Dehn twist}
 \label{fig:Dehn_twist}
\end{figure}

Since $\tau$ is the identity near the boundary of $T^*S^n$, we can extend $\tau$ to a symplectomorphism of $(W,\omega)$: simply extend $\tau$ to be the identity outside the support of $\tau$.

\section{Contact surgery and symplectic handle attachment}
\label{sec:symplectic_handle_attachment}

Let $(M,\xi)$ be a contact manifold, and let $S$ in $M$ be an
isotropic $k$--sphere with a trivialization $\epsilon$ of its
conformal symplectic normal bundle.  Then we can perform contact
surgery along $(S,\epsilon)$.  We shall write the surgered contact
manifold as
\begin{equation*}
  \widetilde {(M,\xi)}_{S,\epsilon}
\end{equation*}
In case of Legendrian surgery, there is no choice for the framing
$\epsilon$, and consequently, we shall drop the framing from the
notation in that case.

We shall now describe a model for contact surgery in terms of symplectic handle attachment. For later computations, we slightly modify Weinstein's original construction, \cite{Weinstein:surgery}.

\subsection{"Flat" Weinstein model for contact surgery}
\label{sec:flat_weinstein}
Here we shall discuss a slightly modified version of the Weinstein model for contact surgery.
Let $(M,\xi=\ker \alpha)$ be a contact manifold and suppose that $S$ is an isotropic $k$--sphere in $(M,\xi)$ with trivial conformal symplectic normal bundle, trivialized by $\epsilon$.
Using this framing $\epsilon$ and a neighborhood theorem, see Theorem~6.2.2 in \cite{Geiges:contacttopology}, we can find a \emph{strict} contactomorphism
\begin{eqnarray*}
\psi:~(\nu(S),\alpha) & \longrightarrow & \left( \R \times T^*S^k\times \R^{2(n-k-1)},dz+p\, dq+\frac{1}{2}(x\, dy-y\, dx) \right),
\end{eqnarray*}
where we regard $\R \times T^*S^k\times \R^{2(n-k-1)}$ as a neighborhood of $\{ 0 \} \times S^k \times \{0 \}$.

\begin{remark}
We should point out that the contactomorphism $\psi$ depends on the trivialization $\epsilon$.
As a result, the entire construction we shall describe now, depends on this choice.
Note that this is unavoidable, since even smoothly the result of surgery depends on the choice of framing.
\end{remark}

A priori, we can only expect a small neighborhood of $S$ to be contactomorphic to a \emph{small} subset of $\R \times T^*S^k\times \R^{2(n-k)}$ via a strict contactomorphism, but we can enlarge this neighborhood by composing with the following non-strict contactomorphism
\begin{eqnarray*}
\phi_C:~\R \times T^*S^k\times \R^{2(n-k)} & \longrightarrow & \R \times T^*S^k\times \R^{2(n-k)}\\
(z,q,p;x,y) & \longmapsto & (Cz,q,Cp;\sqrt C x, \sqrt C y). 
\end{eqnarray*}

Now consider the following model for contact surgery and symplectic handle attachment.
Consider the symplectic manifold $(\R^{2n},\omega_0)$.
We shall use coordinates $(x,y;z,w)$, where there are $n-k-1$ pairs of $(x,y)$ coordinates and $k+1$ pairs of $(z,w)$ coordinates.
The symplectic form is then given by
$$
\omega_0=dx\w dy +dz\w dw.
$$
Note that the vector field
$$
X=\frac{1}{2}(x\partial_x+y \partial_y)+2z\partial_z -w \partial_w
$$
is Liouville for $\omega_0$.

Now consider the set 
$$
S_{-1}:=\{ (x,y,z,w)\in \R^{2n}~|~|w|^2=1 \}.
$$
The Liouville field $X$ is transverse to this set, and induces the contact form
$$
\alpha=\frac{1}{2}(x\, dy-y\, dx)+2z\, dw+w\, dz.
$$
We see that the sphere $\{ (x,y,z,w)~|~x=y=0,~z=0, |w|^2=1 \}\cong S^k$ describes an isotropic sphere in $S_{-1}$ with trivial conformal symplectic normal bundle.
We shall think of $S_{-1}$ as a neighborhood of the isotropic sphere $S$, in other words $S_{-1}$ can be thought of as the situation before surgery.
In fact, the set $S_{-1}$ is a standard neighborhood of an isotropic sphere of dimension $k$ with trivial normal bundle, since we have the following contactomorphism,
\begin{eqnarray*}
\psi_W:\R \times T^*S^{k} \times \R^{2(n-k-1)} & \longrightarrow & S_{-1} \\
(z,q,p,x,y) & \longmapsto & (x,y;zq+p,q).
\end{eqnarray*}
Here we regard the cotangent bundle $T^*S^k$ as a subspace of $\R^{2(k+1)}$ by using coordinates $(q,p)\in \R^{2(k+1)} $, where $q^2=1$ and $q \cdot p=0$. 
Note that $\psi_W$ is a strict contactomorphism,
\begin{eqnarray*}
\psi_W^*(\frac{1}{2}(x\, dy-y\, dx)+2z\, dw, dw+w\, dz) && =\\
\frac{1}{2}(x\, dy-y\, dx)+2(zq+p)dq+q\, d(zq)+q\, dp
&& =\frac{1}{2}(x\, dy-y\, dx)+p\, dq+dz.
\end{eqnarray*}
To see that the latter step holds, use that $q\, dq=0$ and $p\, dq+q\, dp=0$.

We can combine the above three maps to obtain a contactomorphism from $\nu(S)\subset M$ to $S_{-1}$ in the Weinstein model
\begin{eqnarray}
\label{eq:to_neighborhood_in_Weinsteinmodel}
\Phi_C:=\psi_W\circ \phi_C \circ \psi:~\nu(S) &\longrightarrow & S_{-1}.
\end{eqnarray}
This map is not a strict contactomorphism, but since it multiplies the contact form with a constant rather than an arbitrary function, we can adapt the following lemma from \cite{Geiges:contacttopology}, Lemma~5.2.4, for a gluing construction.
\begin{lemma}
\label{lemma:construct_symplectomorphism}
For $i=0,1$, let $(M_i,\alpha_i)$ be a (not necessarily closed) contact type hypersurface in a symplectic manifold $(W_i,\omega_i)$ with respect to the Liouville vector field $Y_i$. 
Suppose $\phi:(M_0,\alpha_0)\to (M_1,\alpha_1)$ is a contactomorphism such that $\phi^*\alpha_1=C\alpha_0$ for some constant $C$.
Then $\phi$ extends to a symplectomorphism between neighborhoods of $M_0$ and $M_1$ by sending flow lines of $Y_0$ to flow lines of $Y_1$.
\end{lemma}
Furthermore, we can choose a large $C$ in Formula~\eqref{eq:to_neighborhood_in_Weinsteinmodel}, which means that we can get arbitrary large neighborhoods in the Weinstein model.
\begin{remark}
We can also adapt the proof of Proposition~3.1 in~\cite{CSK:Minimalatlas} to obtain a contactomorphism from $\nu(S)$ to the full Weinstein model, i.e.~a surjective map to $S_{-1}$.
This contactomorphism is in general not strict, or even admissible for Lemma~\ref{lemma:construct_symplectomorphism}.
Therefore we shall restrict ourselves to a contactomorphism as in Formula~\eqref{eq:to_neighborhood_in_Weinsteinmodel}.
\end{remark}

\subsubsection{Attaching a symplectic handle}
Let us begin by defining a symplectic handle.
The contactomorphism $\Phi_C$ identifies the neighborhood $\nu(S)\subset M$ with a neighborhood of the isotropic sphere in $S_{-1}$.
Suppose that the neighborhood provided by $\psi$ has size
$$
size_\psi(\nu(S)):=\max_{(x,y,z,w)\in \psi(\nu(S))}\sqrt{x^2+y^2+z^2}
=\tilde \epsilon .
$$
Then by choosing $C>2/\tilde \epsilon$ we can ensure that the neighborhood provided by $\Phi_C$ has size larger than $1$, i.e.~the maximal $(x,y,z)$ coordinates are larger than $1$.

We first define the profile for the handle. Fix a small $\delta>0$: this parameter serves as a smoothing parameter. 
Choose smooth functions $f,g:\R_{\geq 0} \to \R$ such that
\begin{itemize}
\item $f$ is increasing.
\item $f(w)=1$ for $w\in [0,1-\delta]$, $f(w)=w+\delta$ for $w>1-\delta/2$.
\item $g$ is increasing.
\item $g(z)=z$ for $z<1$, $g(w)=1+\delta$ for $w>1+\delta$.
\end{itemize}
~\\
See Figure~\ref{fig:functions_profile} for a sketch of these functions.
\begin{figure}[htp]
  \centering
  \includegraphics[width=0.6\textwidth,clip]{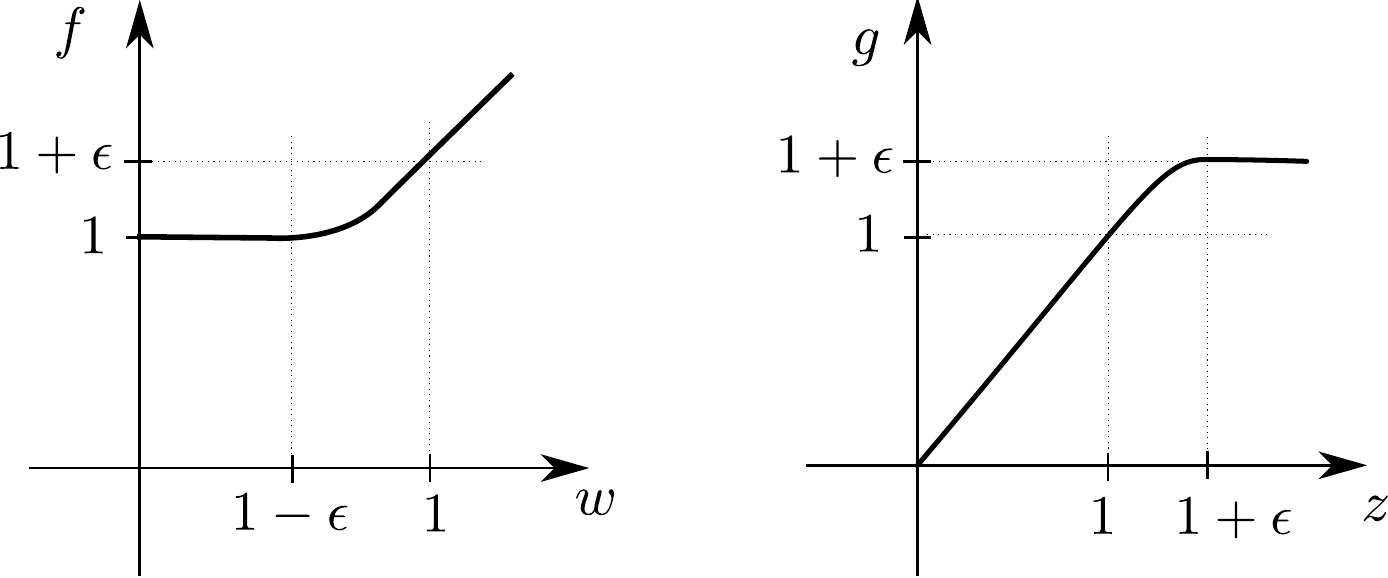}
  
  \caption{Functions for the profile of a symplectic handle}\label{fig:functions_profile}
\end{figure}
~\\

Define
$$
F(x,y,z,w):=-f(w^2)+g(x^2+y^2+z^2).
$$
Define a hypersurface $S_1:=\{ (x,y,z,w)~|~F(x,y,z,w)=0 \}$.
This hypersurface is of contact type, because the Liouville vector field $X$ is transverse to $S_1$,
$$
X(F)=\left( \frac{1}{2}(x^2+y^2)+2z^2\right )g'+w^2f'> 0
$$
for points $x,y,z,w$ such that $F(x,y,z,w)=0$, as points with $g'=0$ are precisely those with $w^2=1$, and points with $f'=0$ are those with $x^2+y^2+z^2=1$.
The hypersurface $S_{1}$ is meant to describe the result of the surgery along $S$.
See Figure~\ref{fig:flat_weinstein} for a sketch of the situation.

\begin{figure}[htp]
  \centering
  \includegraphics[width=0.6\textwidth,clip]{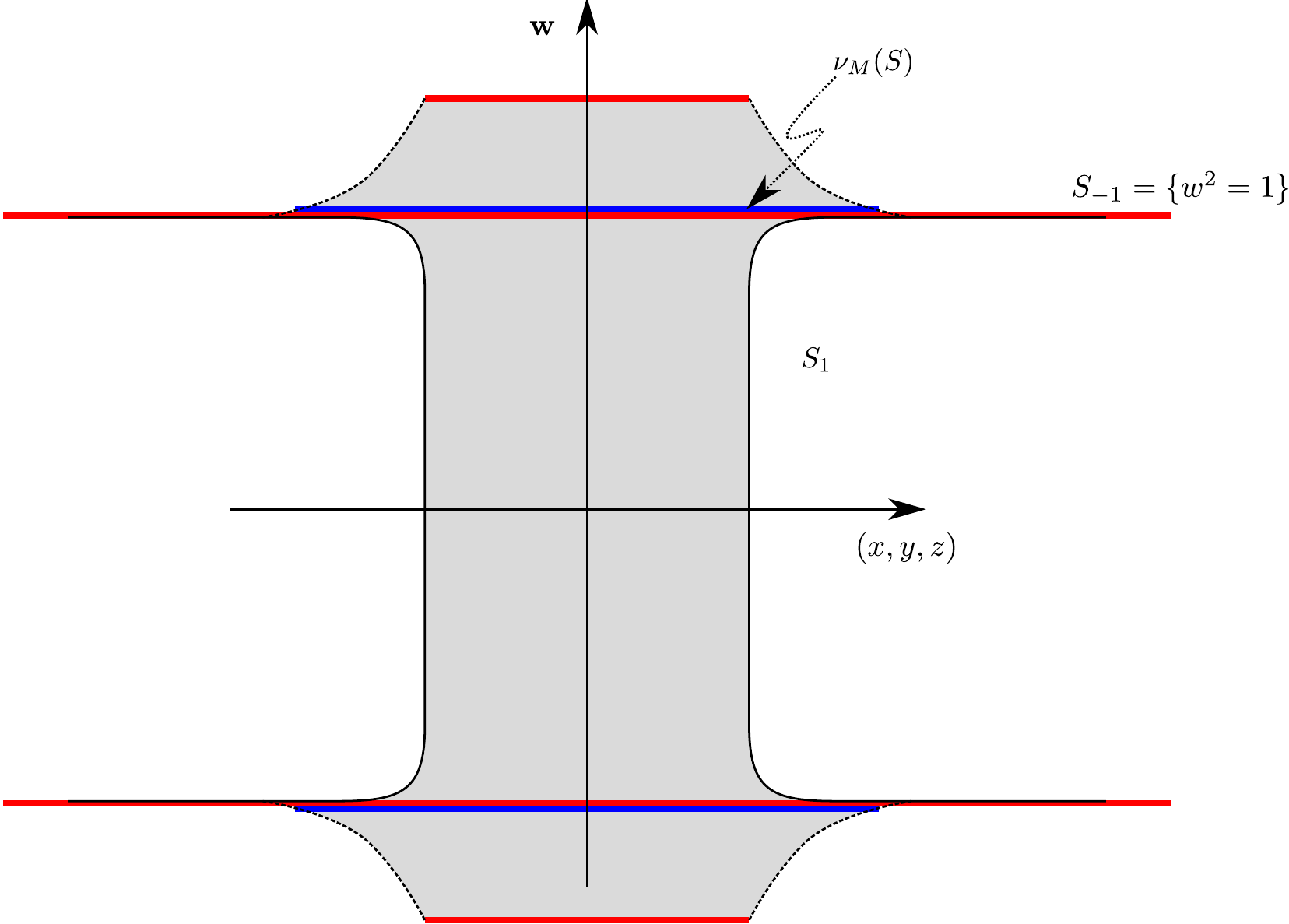}
  
  \caption{A symplectic handle in the flat Weinstein model}\label{fig:flat_weinstein}
\end{figure}

\begin{remark}
\label{rem:profile}
Instead of a profile for a symplectic handle described by the above function $F$, one more commonly chooses a profile of the form
$$
x^2+y^2+z^2-w^2=c.
$$
The advantage is that
$$
G=x^2+y^2+z^2-w^2
$$
defines an $\omega$--convex Morse function with respect to the Liouville field $X$ with one critical point on the handle.
The main reason for preferring $F$ is that it simplifies later computations.
Note that topologically the two profiles are the same.
Furthermore, one can adapt the $\omega$--convex function $G$ to the above profile as well.
See the summary in Proposition~\ref{prop:handle_attachment}.
\end{remark}

In order to describe the surgery, we shall use handle attachment along a symplectic manifold $(W,\omega)$ with contact type boundary $M$.

Define the symplectic handle $(H_{k+1},\omega_0)$ as follows.
$H_{k+1}$ consists of those points $p \in (\R^{2n},\omega_0)$ such that
\begin{itemize}
\item There is a $t\in [0,1]$ such that the time-$t$ flow of $X$ satisfies $Fl^X_t(p) \in \Phi_C(\nu(S))$. This is the gluing part of the symplectic handle.
\item There are $t_1\leq 0$ such that $Fl^X_{t_1}(p) \in \Phi_C(\nu(S))$ and $t_2\geq 0$ such that $Fl^X_t(p) \in S_1$.
\end{itemize}

Let us now attach this symplectic handle $H_{k+1}$ to $(W,\omega)$.
A neighborhood of the boundary of $(W,\omega)$ is symplectomorphic to $([-1,0]\times M,d(e^t,\alpha)$; call this symplectomorphism $\psi_\partial:\nu_W(M)\to [-1,0]\times M$ (note that we can attach a piece of a symplectization of $M$ to $W$ to ensure we have such a neighborhood).
In particular, we have a symplectomorphism
\begin{eqnarray*}
\psi_\partial: \nu_W(\nu_M(S) ) & \longrightarrow & [-1,0]\times \nu_M(S).
\end{eqnarray*}
We can compose this symplectomorphism with the map
\begin{eqnarray*}
\tilde \Phi_C: [-1,0]\times \nu_M(S) & \longrightarrow & H_1 \\
(t,p) & \longmapsto & Fl^X_t( \Phi_C(p) ). 
\end{eqnarray*}
This map is also a symplectomorphism, cf.~Lemma~\ref{lemma:construct_symplectomorphism} (or rescale the symplectic form on $H_{k+1}$).

Now attach the symplectic handle
$$
\tilde W:=W \cup H_{k+1}/\sim.
$$
Here we glue $x$ in $\nu_W(\nu_M(S) ) \subset W$ to $y$ in $H_1$ if and only if $\tilde \Phi_C \circ \psi_\partial(x)=y$.
By Lemma~\ref{lemma:construct_symplectomorphism} the resulting manifold $\tilde W$ is again symplectic and its boundary is a contact manifold that is diffeomorphic to the surgered manifold $\widetilde {(M,\xi)}_{S,\epsilon}$, obtained by performing surgery on $M$ along the isotropic submanifold $S$ with framing $\epsilon$.

\begin{definition}
The above attaching procedure is called {\bf symplectic handle attachment} along $S$ at the convex end of $W$.
We call the attachment {\bf subcritical} if $\dim S<n$ and {\bf critical} if $\dim S=n$. 
The induced operation on the convex end is called {\bf contact surgery} along $S$.
The contact surgery is called {\bf subcritical} if $\dim S<n$ and {\bf critical} or {\bf Legendrian} if $\dim S=n$. 
\end{definition}

\begin{remark}
\label{rem:extending_convexity}
Since we attach a symplectic handle to a cobordism by gluing flow lines of the respective Liouville fields, we see that we can extend the Liouville field defined in a neighborhood of the convex end of $W$ to the new symplectic manifold $(\tilde W,\tilde \omega)$.

As alluded to in Remark~\ref{rem:profile}, there is a function on a symplectic handle that is $\omega_0$--convex for the Liouville field $X$.
By slightly modifying the $\omega$--convex function $f$ on $W$, we can glue this function to the one on the symplectic handle.
This is sketched in Figure~\ref{fig:omega_convex}

\end{remark}

\begin{figure}[htp]
  \centering
  \includegraphics[width=0.6\textwidth,clip]{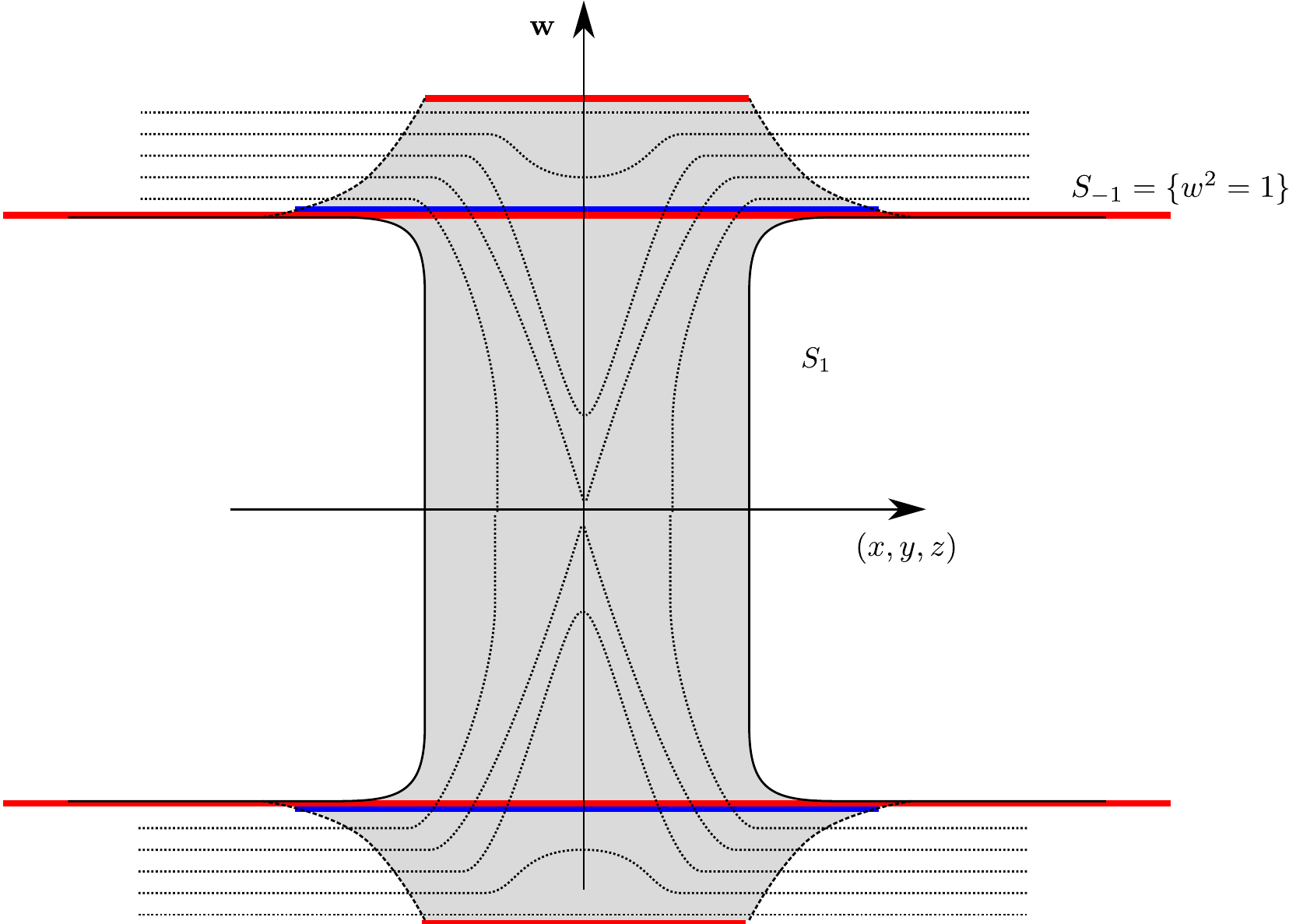}
  
  \caption{A sketch of the modifications to glue $\omega$--convex functions }\label{fig:omega_convex}
\end{figure}

Let us summarize the above discussion in the following proposition.
\begin{proposition}
\label{prop:handle_attachment}
Let $(W,\omega)$ be a symplectic cobordism.
Suppose that $i:S\to \partial W$ is an embedded isotropic $k$--sphere in the convex end of $W$ whose conformal symplectic normal bundle is trivialized by $\epsilon$.

Then we can attach a handle $H_1$ to $W$ along $S$ with framing $\epsilon$ to obtain a symplectic cobordism $(\tilde W,\tilde \omega)$.

Furthermore, if $(W,\omega)$ admits an $\omega$--convex function $f$, then $f$ can be extended to an $\tilde \omega$--convex function $\tilde f$ on $\tilde W$ such that $\tilde f$ has only one additional critical point. 
\end{proposition}

\begin{remark}
We see that we can attach symplectic handles under rather mild assumptions to the convex end of a symplectic manifold.
The converse, i.e.~attaching handles to the concave end of a symplectic manifold is much more restrictive.
Indeed, there are many examples of non-fillable contact manifolds, which illustrates that concave handle attachment has additional requirements.
\end{remark}

\subsection{Symplectic handle cancellation}
\label{sec:symplectic_handle_cancellation}
The main technical tool we shall use is Lemma~3.6b from \cite{Eliashberg:psh}.
Here is a formulation that is suitable for our purposes.
\begin{lemma}[Eliashberg]
Let $(W,\omega)$ be a symplectic manifold and $f$ be an $\omega$--convex function.
 
Let $p$ and $q$ be non-degenerate critical points of $f$ and $d\in ] f(p),f(q)[$
Suppose that
\begin{itemize}
\item $\ind_q(f)=\ind_p(f)+1$.
\item The sphere $S^-_q$, obtained by intersecting the stable manifold $W^s(q)$ with the level set $\{ x~|~f(x)=d \}$, intersects the sphere $S^+_p$, formed by intersecting the unstable manifold $W^u(p)$ with $\{ x~|~f(x)=d \}$, transversely in one point.
\end{itemize}
Then the critical points can be cancelled by a $J$--convex deformation of $f$ in a neighborhood of $[f(p),f(q)]$. 
\end{lemma}
A very similar statement with proof can also be found in~\cite{Cieliebak:Symplectic_geometry_Stein}, Proposition~10.9.

Given the lemma, we can perform symplectic handle cancellation in a way similar to the one in the smooth case, see \cite{Milnor:h-cobordism}, Theorem~5.6.
We shall briefly describe the particular setup which we shall use.
This will be the simplest case of handle cancellation: it can occur after consecutive attachment of handles with index difference~$1$.

Let $(W^{2n}_1,\omega)$ be a symplectic manifold such that $M_1\subset \partial W_1$ is a convex end.
Choose an $\omega$--convex function $f_1$ near the convex end and let $X_1$ be the associated Liouville vector field.

Now suppose that $S_1\subset M_1$ is an isotropic $(n-2)$--sphere with a trivialization $\epsilon$ of its conformal symplectic normal bundle.
Suppose furthermore that $S_1$ bounds a Legendrian $(n-1)$--disk $D_1$ in $M_1$.
Now form the symplectic manifold $(W_2,\omega_2)$ by attaching a symplectic $(n-1)$-handle along $S_1$,
$$
(W_2,\omega_2)=W_1\cup_{S_1,\epsilon}H_{n-1}.
$$
The $\omega_1$--convex $f_1$ can be extended to an $\omega_2$--convex function $f_2$ as mentioned in Remark~\ref{rem:extending_convexity}: this new $\omega_2$--convex function has one additional critical point, corresponding to the middle of the handle.
We shall denote this critical point by $p$.

Note that the convex end $M_1$ is surgered into a new convex end $M_2\subset \partial W_2$.
This convex end comes with a Legendrian $(n-1)$--sphere $S_2$ which is formed as follows.

First observe that there is a parallel copy $D_2$ of the core of $H_{n-1}$ which is a Legendrian $(n-1)$--disk.
More explicitly, we can think of $H_{n-1}$ as $D^2\times T^*S^{n-1}$.
Then put
\begin{eqnarray*}
\phi:~D^{n-1} & \longrightarrow & H_{n-1} =D^2\times T^*D^{n-1} \\
w & \longmapsto & (h_1(|w|)x_0,0,h_2(|w|)w).
\end{eqnarray*}
Here $x_0=(1,0)\in D^2$, and $h_1$ and $h_2$ are functions parametrizing the profile $S_1$.
We see directly that $\alpha=1/2(x\, dy-y\, dx)+2z\, dw+w\, dz$ restricts to $0$ on $D_2:=\phi(D^{n-1})$.
After this, we can match $D_1$ (which is partially removed after the handle attachment of $H_{n-1}$) to glue to $D_2$.
This gives the Legendrian sphere $S_2$.

\begin{remark}
To visualize the handle cancellation that is going to occur in the next step, observe that $S_2$ intersects the belt sphere of $H_{n-1}$ transversely in one point, namely in $\phi(0)$.
\end{remark}

Since $S_2$ is Legendrian, the conformal symplectic normal bundle is trivial, so we can form $W_3$ by critical $n$--handle attachment along $S_2$ without reference to a framing,
$$
(W_3,\omega_3):=W_2 \cup_{S_2} H_{n}.
$$
As before, we can extend the $\omega_2$--convex function $f_2$ to an $\omega_3$--convex function $f$ on $W_3$.
Denote the additional critical point of $f_3$ by $q$.
We shall denote the gradient-like Liouville vector field on $W_3$ by $X_3$.
The convex end $M_2$ is surgered yielding the contact manifold $M_3$.

Now intersect a level set $\{ f_3=d \}$, with $d$ between $f_3(p)$ and $f_3(q)$, with the stable manifold $W^s(q)$ and the unstable manifold $W^u(p)$ to form the spheres $S^-(q)$ and $S^+(p)$, respectively.
These spheres intersect transversely in one point, as we can see from the unique flow line of the Liouville vector field $X_3$ from $p$ to $q$. 

This means that Lemma~3.9 applies, so we can deform $f_3$ to another $\omega_3$--convex function $g_3$ such that $g_3(x)=f_3(x)$ on sublevel sets $\{ f_3<c=f(p)-\delta \}$.
In particular, on such sublevel sets $g_3$ coincides with $f_1$.
Furthermore, $g_3$ has no critical points whenever $g_3(x)\geq c$.
This means that $\{ g_3(x)\geq c \}$ looks like a symplectization, so we conclude that the completion of $W_1$, i.e.~the manifold obtained from $W_1$ by attaching the positive end of a symplectization, is symplectomorphic to the completion of $W_3$.

  \begin{figure}[htp]
    \centering
    \includegraphics[width=0.4\textwidth,clip]{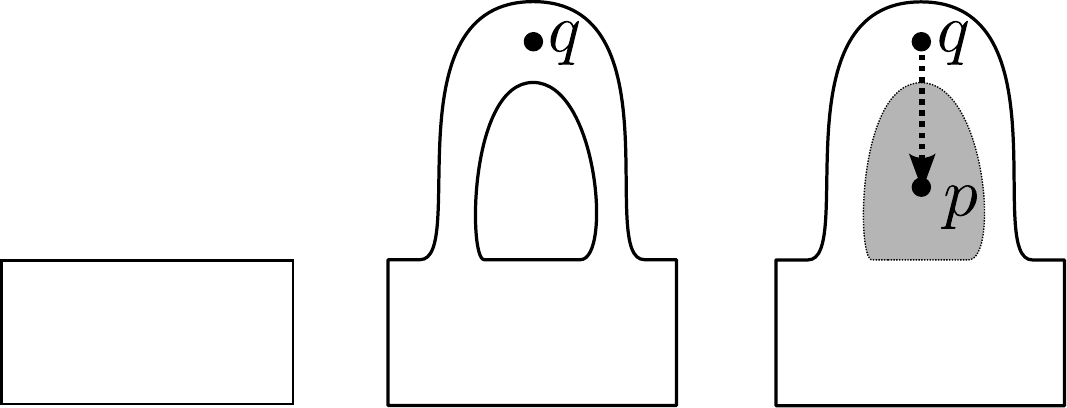}%
    \caption{Cancellation of symplectic handles: the creation of a unique flow line between the critical points $q$ and $p$}\label{fig:cancellation}
  \end{figure}

We summarize the conclusion in the following lemma.
\begin{lemma}[Handle cancellation in successive handle attachment]
\label{lemma:handle_cancel}
Let $(W_1,\omega_1)$ and $(W_3,\omega_3)$ be the symplectic manifolds as formed above by successive handle attachment.
Then the completion of $(W_1,\omega_1)$ is symplectomorphic to the completion of $(W_3,\omega_3)$.
In particular, $M_1$ is contactomorphic to $M_3$.
\end{lemma}

\section{Surgery and open books}

In this section we try to describe some relations between contact
surgery and open books. 
Let us summarize the results that will be proved below.
If an isotropic sphere $S$ lies in
the binding of an open book and if the framing is compatible with the
open book, then subcritical surgery along $S$ can also be described in
terms of handle attachment to the pages of an open book.

On the other hand, critical contact surgery can be regarded as a
change in the monodromy of the open book, if the sphere used for the
surgery lies nicely in a page.

We apply this to show that stabilization of open books leads to
contactomorphic contact manifolds.  The basic strategy is the
following.  To stabilize an open book we attach an $n$--handle to the
$2n$--dimensional page forming a new Lagrangian sphere and change the
monodromy by composing with a right-handed Dehn twist along the newly
formed Lagrangian sphere.

We shall show that the handle attachment to the page can be realized
by a subcritical handle attachment to the convex end of $[-1,0]\times M$ and that the
change in monodromy is realized by a critical handle attachment.  The
latter turns out to cancel the former, so we obtain the same contact
manifold.

\subsection{Subcritical surgery and open books}

Let us first describe the situation for trivial monodromy, since that
situation is more easily visualizable.  Let $\Sigma$ be a compact
Stein manifold with boundary $B := \partial \Sigma$ and consider the
open book
\begin{equation*}
  M := \open (\Sigma,\id) \;.
\end{equation*}
Suppose that $S$ is an isotropic (possibly Legendrian) sphere in $B$
with a trivialization $\epsilon$ of its conformal symplectic normal
bundle.  We can perform contact surgery along $(S,\epsilon)$ giving
rise to a contact manifold $\tilde B$. The associated surgery
cobordism also gives a Stein filling for $\tilde B$, which we will
denote by $\tilde \Sigma$. Alternatively, $\tilde \Sigma$ can be
regarded as the Stein manifold obtained from $\Sigma$ by handle
attachment along $(S,\epsilon)$.

Note that $S$ also gives rise to an isotropic submanifold of $M$.
Indeed, we have an isotropic sphere in the binding: take $S_M :=
S\times \{ 0 \}\subset B\times D^2$.  Since we have the following
contact form near the binding,
\begin{equation*}
  \lambda + x\,dy - y\,dx \;,
\end{equation*}
we see that we also get a trivialization of the conformal symplectic
normal bundle of $S_M\subset M$, given by $\epsilon_M :=
\epsilon\oplus \langle \partial_x, \partial_y\rangle$.  For later use,
it is useful to give the last factor a name,
\begin{equation*}
  \epsilon_{D^2} = \langle \partial_x, \partial_y\rangle \;.
\end{equation*}

Contact surgery on $M$ along $(S_M,\epsilon_M)$ gives the subcritical
fillable contact manifold $\widetilde M:=\partial (\Sigma \times
D^2)$, as we can see from performing handle attachment to the filling
$\Sigma\times D^2$ of $M$.  On the level of open books, we see that
$\widetilde M$ has a supporting open book with page $\widetilde
\Sigma$ and the identity as monodromy.

This setup also describes the general situation, since the surgery
takes place near the binding.  As a result, we have the following
proposition,

\begin{proposition}\label{prop:subcritical}
  Let $\Sigma$ be a compact Stein manifold with boundary $B$ and let
  $\psi$ be a compactly supported symplectomorphism such that
  \begin{equation*}
    M := \open(\Sigma,\psi)
  \end{equation*}
  is a contact open book.  Suppose that $S_B$ is an isotropic
  $(k-1)$--sphere in the binding $B$ with a trivialization $\epsilon$
  of its conformal symplectic normal bundle in $B$.  Then there is a
  corresponding isotropic $(k-1)$--sphere $S_M\subset M$ with
  trivialization $\epsilon\oplus \epsilon_{D^2}$ of its conformal
  symplectic normal bundle such that
  \begin{equation*}
    \open(\Sigma\cup_{S_B} \text{$k$--handle}, \psi \cup_{S_B}\id)
    \cong  \widetilde{\open(\Sigma,\psi)}_{(S_M, \epsilon \oplus
      \epsilon_{D^2})} \;.
  \end{equation*}
\end{proposition}

In other words, this kind of subcritical surgery is realized by handle
attachment to the page of an open book without changing the monodromy.

\subsection{Critical surgery and open books}

Now consider a contact open book $M:=\open(\Sigma,\psi)$ having a
Lagrangian sphere $L_S$ in a page.  We can assume that $L_S$
represents a Legendrian sphere in $\Sigma\times \R$, or in
other words in the contact open book $M$.

\begin{lemma}\label{lemma:lagsphere_legsphere}
  Let $M^{2n+1}= \open(\Sigma,\psi)$ be a contact manifold of
  dimension greater than $3$.  If $L_S$ is a Lagrangian sphere in the
  page of a contact open book $M^{2n+1}$, then we can isotope the contact structure on $M^{2n+1}$ and find a supporting open book with symplectomorphic page and isotopic monodromy such that $L_S$ becomes Legendrian in $M^{2n+1}$.
\end{lemma}
\begin{proof}
  Suppose $\lambda$ is a primitive of the symplectic form $\omega$ on
  $\tilde \Sigma$.  Then on a Weinstein neighborhood of $L_S$ we can
  find a primitive
  \begin{equation*}
    \lcan =  p\, d q
  \end{equation*}
  of the symplectic form $\omega$, where $( q, p)$ are
  coordinates on $T^*S^n\cong \nu(L_S)$.  Since $\lambda -\lcan$ is
  closed, we can find a function $g$ such that
  \begin{equation*}
    \lambda - \lcan = dg \;,
  \end{equation*}
  because $H^1_{dR}(S^n)=0$ as $n>1$. Now put
  \begin{equation*}
    \widetilde \lambda := \lambda - d(\rho g)\;,
  \end{equation*}
  where $\rho$ is a smooth cut-off function that is $1$ in a
  neighborhood of $L_S$ with support in the Weinstein neighborhood
  $\nu(L_S)$.  
  Note that $d\widetilde \lambda=d\lambda$ is still symplectic.

  On the Lagrangian sphere $L_S$, $\widetilde \lambda$
  vanishes, so $L_S$ lies in the kernel of the contact form $dt +
  \widetilde \lambda$, so it is Legendrian. Furthermore, the associated contact structure is isotopic to the one we started with by Gray stability.
 
\end{proof}

\begin{remark}
  In dimension $3$, every curve in a page is Lagrangian, but to
  realize a curve as a Legendrian one needs to perturb transversely to
  a page.  Hence we cannot directly formulate an analogue to
  Lemma~\ref{lemma:lagsphere_legsphere}.  On the other hand, in
  dimension $3$ one can always find a supporting open book such that a
  Legendrian lies in a page, see \cite{Geiges:contact_geometric_topology}, section~4, for a discussion of the $3$-dimensional situation.
\end{remark}

Given the Lagrangian sphere $L_S$, we get a compactly supported
symplectomorphism $\tau_{L_S}$, a right-handed Dehn twist along $L_S$.
We can now change the monodromy of the contact open book by adding
Dehn twists along $L_S$, but we can also perform contact surgery along
$L_S$.  We shall now show that these operations coincide.

\subsubsection{Surgery and monodromy}

The goal of this section is to provide a proof of the following folk theorem. This result is well known in dimension $3$, \cite{Geiges:contact_geometric_topology}.

\begin{theorem}\label{thm:Dehntwist_legendrian_surgery}
  Let $\open(\Sigma,\psi)$ be a contact open book with Legendrian
  sphere $L_S$, that restricts to a Lagrangian sphere in $\Sigma$.
  Denote the contact manifold obtained from $\open(\Sigma,\psi)$ by
  Legendrian surgery along $L_S$ by
  $\widetilde{\open(\Sigma,\psi)}_{L_S}$.  Then the contact manifolds
  \begin{equation*}
    \open(\Sigma,\psi\circ \tau_{L_S}) \cong
    \widetilde{\open(\Sigma,\psi)}_{L_S}
  \end{equation*}
  are contactomorphic.
\end{theorem}

\begin{proof}
  The proof has two steps.  In the first step we shall show that
  Legendrian surgery on $\open(\Sigma,\psi)$ along $L_S$ yields a
  contact manifold with a supporting open book $(B,\tilde \theta_r)$,
  where $B$ is the binding and $\tilde \theta_r$ the map to $S^1$.
  The new supporting open book has the same page as the supporting
  open book before the surgery, and we can localize the monodromy.  In
  the second step we determine how the monodromy is changed.

  \noindent
  \textbf{Step 1: Supporting open book after surgery}\\
  The contact open book $\open(\Sigma,\psi)$ gives rise to a contact manifold $(M,\xi)$ with a supporting open book $(B,\theta)$, where $B$ is a codimension $2$ submanifold of $M$ and $\theta:\, M-B \to S^1$ a fiber bundle over $S^1$.
  We think of $L_S$ as a Legendrian submanifold of $(M,\xi)$ lying in the page $0\in \R/ \Z\cong S^1$.
  Choose a neighborhood $\nu(L_S)$ of $L_S$ such that $\nu(L_S)$ is contactomorphic to $T^*L_S \times \R$, and such that $\nu(L_S)\subset \theta^{-1}(]-\epsilon,\epsilon[)$.
  In particular, we can restrict the map $\theta$ to a map $\theta_r:=\theta|_{\nu(L_S)}:\, \nu(L_S)\to ]-\epsilon,\epsilon[$.

  Note that $\nu(L_S)$ can be identified with a neighborhood of $\{
   z=0 \}$ in the hypersurface of contact type $\{  w^2=1 \}
  \subset \C^n$, as described before in the interlude on the ``flat''
  Weinstein model.  Let us use the identification to choose a specific
  model for $\theta_r$.

  \begin{figure}[htp]
    \centering
    
    \includegraphics[width=0.6\textwidth,clip]{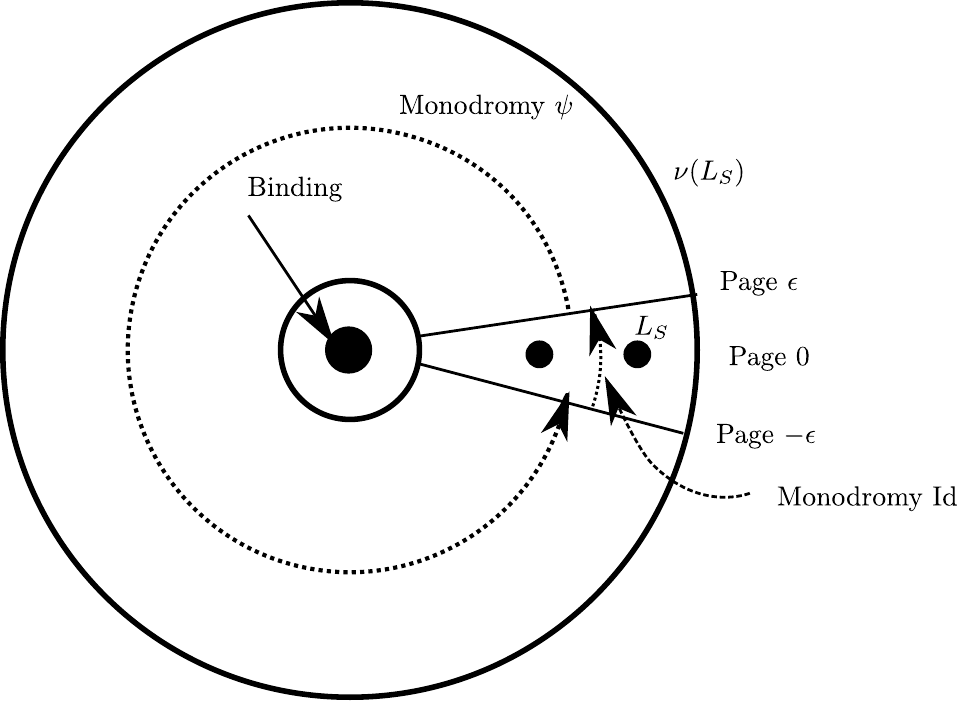}%
    
    \caption{Pages around a Legendrian sphere}
    \label{fig:Legendrian_in_pages}
  \end{figure}

  By isotoping the open book and Gray stability we can assume that the
  restricted map $\theta_r$ has the form
  \begin{eqnarray*}
    \theta_r:~\nu(L_S) & \longrightarrow & ]-\epsilon,\epsilon[ \\
    ( z, w) & \longmapsto &  z \cdot  w
  \end{eqnarray*}
  Indeed, since the Reeb field on $\{  w^2=1 \}$ is given by
  \begin{equation*}
    R_B = w\,\partial_z \;,
  \end{equation*}
  we see that $R(\theta_r)>0$, so $\theta_r$ gives also a supporting
  open book for $(M,\xi)$.

  Next perform Legendrian surgery along $L_S$ using the ``flat'' Weinstein model as described in Section \ref{sec:flat_weinstein}.
This means that we remove a neighborhood $\nu_{M,\epsilon}(L_S)$ of $L_S$ and glue in the set $S_{1}$, which is the zero set of the function $F(z;w)=-f(w^2)+g(z^2)$.
In our setup, the Reeb field on $S_{1}$ is a positive multiple of the Hamiltonian vector field of $F$, given by
$$
X_F=\frac{\partial F}{\partial z}\partial_w -\frac{\partial F}{\partial w}\partial_z=2g'z \partial_w+2f'w \partial_z.
$$
By our choice of the functions $f$ and $g$ we see that
$$
R(z\cdot w)=NX_F(z\cdot w)=N \left( 2g' |z|^2+2f'|w|^2 \right)
>0,
$$
so the function $\theta_r$ also defines a suitable open book projection on $S_1$.

\noindent
\textbf{Step 2: Monodromy}\\
  Let us now investigate the monodromy.  The change of the monodromy
  can be localized in an $\epsilon$--neighborhood of page $0$, and
  furthermore this change of monodromy does not depend on the choice
  of $L_S$, since we have described the entire setup with the
  Weinstein model.  So we see that
  \begin{equation*}
    \widetilde{\open(\Sigma,\psi)}_{L_S}
    \cong \open(\Sigma,\psi\circ \psi_{L_S}) \;,
  \end{equation*}
  where $\psi_{L_S}$ is the change in monodromy.  Hence we only need
  to see what Legendrian surgery does to the monodromy in a single
  model situation to determine $\psi_{L_S}$.

\noindent
\textbf{Monodromy before surgery}
Let us first compute the monodromy from page $-\epsilon$ to page $\epsilon$ before the surgery.
Take a point $(z_{-\epsilon},w_{-\epsilon})$ in $A$, i.e.~$|w_{-\epsilon}|^2=1$.
Since we start at page $-\epsilon$, we have $\theta(z_{-\epsilon},w_{-\epsilon})=-\epsilon$, so we may write
$$
z_{-\epsilon}=-\epsilon w_{-\epsilon}+r_{-\epsilon},
$$
where $r_{-\epsilon}\cdot w_{-\epsilon}=0$. In particular $(w_{-\epsilon},r_{-\epsilon})$ can be seen as points in $T^*S^n$.

The Reeb flow is the reparametrized Hamiltonian flow for the Hamiltonian $w^2=1$.
Hence we have the following solutions to the flow equations:
$$
z(s)=(s-\epsilon)w_{-\epsilon}+r_{-\epsilon}, ~~w(s)=w_{-\epsilon}.
$$
We see that the Reeb flow transports the point $(z_{-\epsilon},w_{-\epsilon})$ after time-$s=2\epsilon$ to page $+\epsilon$.
Since the decomposition $z=sw+r$ is preserved, we conclude that the monodromy is trivial.

\noindent
\textbf{Monodromy after surgery}
Let us now consider the monodromy after surgery.
We need to compare with the monodromy before the surgery and for that we shall use the following map to identify subsets of $S_{-1}$ with subsets of $S_{1}$.
Take a point $(z_{-\epsilon},w_{-\epsilon})$ with $|w_{-\epsilon}|^2=1$, but $z_{-\epsilon}\neq 0$.
This point can be seen as a point in $S_{-1}$.
For $a>0$ we define the Liouville vector field
$$
X_a:=(1+a)z \partial_z -a w\partial_w.
$$
Later computations will be simpler when we take the limit $a\to \infty$.
For now we shall use the vector field $X_a$ though for some fixed $a$.

Let us now identify $S_{-1}$ with the flat parts of $S_1$, i.e.~those subsets of $S_1$ where either $f'$ or $g'$ is zero.
The time-$t$ flow of $X_a$ sends $(z_{-\epsilon},w_{-\epsilon})$ to 
$$
(e^{(1+a)t}z_{-\epsilon},e^{-at}w_{-\epsilon}).
$$
So if we send $(z_{-\epsilon},w_{-\epsilon})$ to a flat piece of $S_{1}$ we obtain
$$
(z_{-\epsilon}',w_{-\epsilon})=\left( \frac{1}{|z_{-\epsilon}|} z_{-\epsilon},
|z_{-\epsilon}|^{-\frac{a}{1+a}} w_{-\epsilon}
\right)
.
$$
Let us now consider the limit $a\to \infty$ to enable explicit computations.
We can then later argue using isotopies that our final answer remains valid.
The points on the flat piece are now
$$
(z_{-\epsilon}',w_{-\epsilon}')=\left( \frac{1}{|z_{-\epsilon}|} z_{-\epsilon},
|z_{-\epsilon}| w_{-\epsilon}
\right)
.
$$

Let us now flow with the Reeb vector field to page $\epsilon$.
As stated before, the Reeb vector field is a reparametrization of the Hamiltonian vector field
$$
X_F=2g'z \partial_w+2f'w \partial_z.
$$
On the flat piece of $S_1$ we are interested in, we have
$$
X_F=2z\partial_w.
$$
The time-$s$ flow sends the point $(z_{-\epsilon}',w_{-\epsilon}')$ to
$$
(z'(s),w'(s))=(z_{-\epsilon}',w_{-\epsilon}'+2sz_{-\epsilon}')=
\left( \frac{1}{|z_{-\epsilon}|} z_{-\epsilon},
|z_{-\epsilon}| w_{-\epsilon}+2s\frac{1}{|z_{-\epsilon}|} z_{-\epsilon}
\right)
.
$$
The page number can be computed as
$$
\theta(z'(s),w'(s))=-\epsilon+2s.
$$
Hence we see that at page $\epsilon$ we have arrived at point
$$
(z_\epsilon',w_\epsilon')=\left( \frac{1}{|z_{-\epsilon}|} z_{-\epsilon},
|z_{-\epsilon}| w_{-\epsilon}+2\epsilon\frac{1}{|z_{-\epsilon}|} z_{-\epsilon}
\right).
$$
We transform this point back to $S_{-1}$ to compare with the monodromy before surgery.
We find
$$
(z_\epsilon,w_\epsilon)=\left(
|w_\epsilon'| z_\epsilon' , \frac{1}{|w_{\epsilon}'|} w_{\epsilon}'
\right)
.
$$
Note that $|w_\epsilon'|=\sqrt{ |z_{-\epsilon}|^2+4\epsilon z_{-\epsilon}\cdot w_{-\epsilon}+4\epsilon^2}=|z_{-\epsilon}|$, so we see that the monodromy corresponding to Legendrian surgery is given by
$$
(z_\epsilon,w_\epsilon)=\left(
z_{-\epsilon} , w_{-\epsilon}+2\epsilon \frac{z_{-\epsilon}}{|z_{-\epsilon}|^2}
\right).
$$

\noindent
{\bf Recognizing the monodromy as a Dehntwist}
To interpret the monodromy as a Dehn twist, recall the contactomorphism from Section~\ref{sec:flat_weinstein}
\begin{eqnarray*}
\psi_W:\R \times T^*S^n & \longrightarrow &S_{-1}\\
(z;q,p) & \longmapsto & (zq+p,q),
\end{eqnarray*}
where we regard $T^*S^n$ as a subspace of $\R^{2n+2}$ using coordinates $(q,p)$ subject to the relations $q^2=1$ and $q\cdot p=0$.

Using this description, we decompose $(z_\epsilon,w_\epsilon)$ as
$$
z_\epsilon=\epsilon w_\epsilon+r_\epsilon,
$$
where $|w_\epsilon|^2=1$ and $w_\epsilon \cdot r_\epsilon=0$.
We have a similar decomposition for the initial point,
$$
z_{-\epsilon}=-\epsilon w_{-\epsilon}+r_{-\epsilon}.
$$ 
Let us now compute the $(w_\epsilon,r_\epsilon)\in T^*S^n$ in terms of $(w_{-\epsilon},r_{-\epsilon})$ to compare with Dehn twists.
We have 
\begin{eqnarray*}
w_\epsilon && =w_{-\epsilon}+2\epsilon \frac{-\epsilon w_{-\epsilon}+r_{-\epsilon}}{r_{-\epsilon}^2+\epsilon^2}  =\frac{r_{-\epsilon}^2-\epsilon^2}{r_{-\epsilon}^2+\epsilon^2}w_{-\epsilon}+\frac{2\epsilon}{r_{-\epsilon}^2+\epsilon^2} r_{-\epsilon}.
\end{eqnarray*}
and
\begin{eqnarray*}
r_\epsilon && =z_\epsilon-\epsilon w_\epsilon=-\epsilon w_{-\epsilon}+r_{-\epsilon}-\epsilon w_{-\epsilon}-\frac{2\epsilon^2 r_{-\epsilon}}{r_{-\epsilon}^2+\epsilon^2}+\frac{2\epsilon^3 w_{-\epsilon}}{r_{-\epsilon}^2+\epsilon^2} \\
&& =-\frac{2 \epsilon r_{-\epsilon}^2}{r_{-\epsilon}^2+\epsilon^2}w_{-\epsilon}+
\frac{r_{-\epsilon}^2-\epsilon^2}{r_{-\epsilon}^2+\epsilon^2}r_{-\epsilon}
\end{eqnarray*}

Now observe that the functions $\frac{\epsilon^2-r_{-\epsilon}^2}{r_{-\epsilon}^2+\epsilon^2}$ and $\frac{2 \epsilon |r_{-\epsilon}|}{r_{-\epsilon}^2+\epsilon^2}$ are the standard functions for the rational parametrization of the circle, i.e.~we can regard them as $\cos g(r_\epsilon)$ and $\sin g(r_\epsilon)$ of some increasing function $g(r_\epsilon)$, respectively.
With this in mind, we see that
$$
\left(
\begin{array}{c}
w_\epsilon \\
r_\epsilon
\end{array}
\right)
=
\left(
\begin{array}{cc}
-\cos g(r_\epsilon) & \sin g(r_\epsilon)/|r_{-\epsilon}| \\
-\sin g(r_\epsilon) |r_{-\epsilon}| & -\cos g(r_\epsilon)
\end{array}
\right)
\left(
\begin{array}{c}
w_{-\epsilon} \\
r_{-\epsilon}
\end{array}
\right)
.
$$
Note that we can recognize this map as a \emph{right-handed} Dehn twist $\sigma_{\tilde g(|p|)}$ if we define $\tilde g(x):=\pi-g(x)$.

\noindent
{\bf Isotopy to correct the map}
Note that we have made two approximations in the above computation.
\begin{itemize}
\item We have ignored the rounding piece of size $\delta$ in $S_1$.
\item We have used $X_\infty$ rather than $X_a$ for some $a\in \R$.
\end{itemize}
Let us argue that these approximations can be corrected for by performing isotopies with compact support on $T^*S^n$.
Indeed, note the following
\begin{itemize}
\item Fix the smoothing parameter $\delta>0$ in $S_1$.
Observe that we miss the monodromy of points $(z,w)$ such that
$$
1-\delta <|z|^2,|w|^2<1+\delta.
$$ 
Note that we can bound the effect of the flow of $X_F$ by $C\delta$ for these points.
More precisely, if we follow the above scheme, but send points in $S_{-1}$ to $S_1$ rather than just into the flat pieces of $S_{-1}$, then we have the following bounds: 
$$
|w_{\epsilon}-w_{-\epsilon}|<C\delta,~|r_{\epsilon}-r_{-\epsilon}|<C\delta.
$$
In particular, we can choose $\delta$ small enough to see that the correction in the round piece can be isotoped to the identity by an isotopy with compact support in $T^*S^n$.
\item By choosing $a$ sufficiently large, the map obtained from the above procedure with the Liouville vector field $X_a$ rather than $X_\infty$ is arbitrarily close to the one obtained with $X_\infty$.
Hence we also find an isotopy to correct for this.
\end{itemize}

 It follows that the monodromy obtained by the surgery is indeed a right-handed Dehn twist.
\end{proof}

An immediate corollary of the above theorem is the following
well-known statement about the relation between fillability and the
monodromy of an open book.

\begin{corollary}
  Let $M:=\open(\Sigma,\psi)$ be a contact open book with a Stein page
  $\Sigma$ and a monodromy that is the product of right-handed Dehn
  twists.  Then $M$ is Stein fillable.
\end{corollary}
\begin{proof}
  The contact manifold $\tilde M:=\open(\Sigma,\id)$ is Stein-fillable
  with filling $\tilde W:= \Sigma\times D^2$.  By
  Theorem~\ref{thm:Dehntwist_legendrian_surgery}, we see that we
  obtain $M$ from $\tilde M$ by critical surgery along Legendrian
  spheres.  Since this can also be done on cobordism level, we obtain
  a Stein filling for $M$ by attaching critical handles along
  Legendrian spheres to $\tilde W$.
\end{proof}

\begin{remark}
The actual Stein filling depends on the precise factorization of the monodromy into right-handed Dehn twists.
Different factorizations of a given monodromy can give rise to distinct Stein fillings for the same contact manifold which is determined by the monodromy itself rather than its factorization into Dehn twists.
\end{remark}

\subsection{Stabilization}

Let us now consider the contact open book $M = \open(\Sigma^{2n},
\psi)$ and suppose that $L$ is an embedded Lagrangian disk $D^n$ in
the page $\Sigma$ whose boundary $\partial L$ is a Legendrian sphere
in $\partial \Sigma$.

As in the previous section we define a new contact open book by
\begin{equation*}
  \widetilde M=\open(\widetilde \Sigma,\widetilde \psi) \;,
\end{equation*}
where $\tilde \Sigma$ is obtained from $\Sigma$ by $n$--handle
attachment along $\partial L$.  The monodromy $\widetilde \psi$
restricts to the identity on the attached $n$--handle and coincides
with $\psi$ on $\Sigma$.  Note that $\tilde \Sigma$ contains a
Lagrangian sphere $L_S$ spanned by the Lagrangian disk $L$ and the
core of the $n$--handle.

\begin{remark}
\label{rem:intersection_belt_sphere}
Let us interpret the critical attachment of a handle $h$ to the page $\Sigma$ of $M$ as subcritical handle attachment to the symplectic cobordism $W:=[0,1]\times M$ as in Proposition~\ref{prop:subcritical}.
Denote the result of this handle attachment by $\tilde W$.

We see that $L_S \subset \tilde W$ intersects the belt sphere of the attached handle transversely in one point, since the part of $L_S$ in the handle has the form $core \times \{  p\}\subset h\times D^2$.
\end{remark}

\begin{definition}
  Let $M = \open(\Sigma,\psi)$ be a contact open book book with an
  embedded Lagrangian disk $L$ as above.
The contact open book
  \begin{equation*}
    \widetilde M:=\open(\widetilde \Sigma,\widetilde \psi\circ
    \tau_{L_S})
  \end{equation*}
  is called the {\bf stabilization} of $\open(\Sigma,\psi)$ along $L$.
\end{definition}

The following claim is a well-known statement due to Giroux.

\begin{proposition}[Giroux]
  The stabilization of a contact open book $\open(\Sigma, \psi)$ along
  a Lagrangian disk $L$ bounding a Legendrian sphere in $\partial
  \Sigma$ is contactomorphic to the contact manifold
  $\open(\Sigma,\psi)$.
\end{proposition}

\begin{proof}
  Let $M^{2n+1} = \open(\Sigma,\psi)$ be a contact open book and $L$ a
  Lagrangian disk $D^n$ in a page $\Sigma$.  To stabilize the open
  book, we first need to attach an $n$--handle to $\Sigma$ along
  $\partial L$ to obtain a new page $\widetilde \Sigma$.  The
  submanifold $\partial L$ of the binding corresponds to an isotropic
  sphere $S$, so on the level of contact manifolds, the first step of
  stabilizing is realized by performing contact surgery along the
  isotropic sphere $S$ as in Proposition~\ref{prop:subcritical}.  In
  the language of symplectic cobordisms, we start with a compact piece
  of the symplectization of $M$, say $[0,1]\times M$, and then attach
  an $n$--handle to $\{1\} \times M$ along $\{1\} \times S$.

  The next step of the stabilization consists of changing the
  monodromy; we have a Lagrangian sphere in the new page $\widetilde
  \Sigma$, which we denote by $L$.  Note that we can assume that $L$
  is also a Legendrian sphere in $\open(\widetilde \Sigma, \widetilde
  \psi)$ by Lemma~\ref{lemma:lagsphere_legsphere}.  The stabilization
  is given by
  \begin{equation*}
    M_S = \open(\tilde \Sigma,\psi \circ \tau_L) \;.
  \end{equation*}
  On the level of contact manifolds, this change of monodromy can be
  realized by performing Legendrian surgery along $L$, as we can apply
  Theorem~\ref{thm:Dehntwist_legendrian_surgery}.  In cobordism
  language, this amounts to attaching an $(n+1)$--handle to $W$ along
  the Legendrian sphere $L$, as described in Section~\ref{sec:symplectic_handle_attachment}.

By construction of the stabilization, the Legendrian sphere $L$ intersects the belt sphere of the previously attached $n$--handle in precisely one point, see Remark~\ref{rem:intersection_belt_sphere}.
This means that the interpretation of the stabilization procedure in terms of symplectic handle attachment fits exactly the description of symplectic handle cancellation in Section~\ref{sec:symplectic_handle_cancellation}.
Hence we apply the handle cancellation lemma~\ref{lemma:handle_cancel} to see that the stabilization yields a contactomorphic manifold.
\end{proof}

\bibliographystyle{amsplain}
\bibliography{../bibliography/biblio}

\end{document}